\documentclass[a4paper,12pt]{amsart}
\usepackage[colorlinks=true,urlcolor=blue,
citecolor=red,linkcolor=blue,linktocpage,pdfpagelabels,
bookmarksnumbered,bookmarksopen]{hyperref}
\usepackage[active]{srcltx}
\usepackage{verbatim}
\usepackage{epsfig,graphicx,color,mathrsfs}
\usepackage{graphicx}
\usepackage{amsmath,amssymb,amsthm,amsfonts}
\usepackage{amssymb}
\usepackage[english]{babel}
\usepackage[scrtime]{prelim2e}
\usepackage[hyperpageref]{backref}
\usepackage[left=2.7cm,right=2.7cm,top=3cm,bottom=3cm]{geometry}

\newtheorem{thm}{Theorem}[section]

\newtheorem{lem}[thm]{Lemma}

\newtheorem{rem}[thm]{Remark}
\newtheorem{defn}[thm]{Definition}

\newcommand{\R}{\mathbb R}

\newcommand{\n}{\nabla}
\newcommand{\la}{\langle}
\newcommand{\ra}{\rangle}

\newcommand{\D}{\hbox{D}}

\begin{document}

\parindent 0pc
\parskip 6pt
\overfullrule=0pt

\title[Rigidity results]
{Rigidity results via a Poincaré type formula for stable solutions of quasilinear anisotropic elliptic equations}

\author{Berardino Sciunzi$^{*}$ and Domenico Vuono$^{*}$}
\address{$^{*}$Dipartimento di Matematica e Informatica, UNICAL, Ponte Pietro  Bucci 31B, 87036 Arcavacata di Rende, Cosenza, Italy}
\email{sciunzi@mat.unical.it, domenico.vuono@unical.it}
\thanks{The authors are supported by PRIN PNRR P2022YFAJH \emph{Linear and Nonlinear PDEs: New directions and applications} and by Gruppo Nazionale per l'Analisi Matematica, la Probabilità. B. Sciunzi and D. Vuono have been partially supported by \emph{INdAM-GNAMPA Esistenza, regolarità e proprietà qualitative per problemi non lineari} E5324001950001.}

\date{\today}

\date{\today}

\begin{abstract}
In this paper we prove some one-dimensional symmetry results for stable solutions to quasilinear anisotropic elliptic equations in $\R^2$ and $\R^3$. In this regard, we give an extension of a formula obtained in \cite{SZ1,SZ2, FSV}.
The techniques used are inspired by the works of \cite{Fhdr,FSV}.
\end{abstract}

\maketitle

\section{Introduction}

We consider $u\in C^{1}(\R^N)$ weak solutions to the problem 
\begin{equation}\label{equazione}
-\operatorname{div}\left(B'(H(\nabla u))\nabla H (\nabla u)\right)=f(u) \quad \text{on }\R^n,
\end{equation}
As usual the weak formulation of \eqref{equazione} is given by
\begin{equation}\label{formulazionedebole}
    \int_{\R^N} B'(H(\nabla u))\langle \nabla H(\nabla u),\nabla \varphi\rangle =\int_{\R^N}f(u)\varphi, \quad \varphi \in C^1_c(\R^N).
\end{equation}

Hereafter, $H\in C^{3,\beta}_{loc}(\R^N \setminus \{0\})$ is a Finsler norm (see Section $2$ for details),  

$f\in C^{1,\beta}_{loc}(0,+\infty)$ and $B$ satisfies
\begin{itemize}
    \item[(h$_B$)]
    \begin{enumerate}
        \item[(i)] $B \in C^{3,\beta}_{loc}(0,+\infty)\cap C^1([0,\infty))$; 

        \item[(ii)] $B(0)=B'(0)=0,\quad B(t), B'(t), B''(t)>0 \quad \forall t \in (0,+\infty)$;

        \item[(iii)] there exist $p>1$, $\kappa\in [0,1)$, $\gamma > 0$, $\Gamma > 0$ such that
        \begin{equation}\label{hp:boundAboveBelow}
            \begin{split}
               \gamma (\kappa+t)^{p-2}t \leq & B'(t) \leq \Gamma(\kappa+ t)^{p-2}t \quad \forall\,t>0\\
                \gamma(p-1)(\kappa+t)^{p-2} \leq & B''(t) \leq \Gamma (p-1) (\kappa+t)^{p-2}\quad \forall\,t>0\,.
            \end{split}
        \end{equation}
    \end{enumerate}
\end{itemize}
For $H(\xi)=|\xi|$ and $B(t)=\frac{t^{p}}{p}$, the operator at left-hand side of \eqref{equazione} boils down to the usual $p-$Laplace operator.

Note that under those assumptions the considered problem is well defined\footnote{The function $B\circ H$ is of class $C^1(\R^N)$, see \cite[Lemma $2.3$]{CoFaVa2}}.
Under assumption $(h_B)$-$(iii)$, by standard regularity results, see \cite{DB,T,CoFaVa1,CoFaVa2}, weak solutions of the equation \eqref{formulazionedebole} belong to $C^{1,\alpha}_{loc}(\R^N)\cap C^2(\R^N\setminus Z_u)$, where $Z_u$ is the critical set where the gradient of a solution vanishes. We refer to \cite{Alberto,Alberto2,DS1,CaRiSc} for results involving the stress field of the solutions.

We define the following matrix
$A: \R^N\setminus \{0\}\rightarrow \operatorname{Mat}(N\times N)$ by 
\begin{equation}\label{matrice}
    A_{ij}(\xi):=B''(H(\xi))\partial_iH(\xi)\partial_jH(\xi)+B'(H(\xi))\partial_{ij} H(\xi),
\end{equation}
for $i,j=1,...,N$

For $k=1,2$, we set 
\begin{equation}
    \lambda_1(t):=B''(t), \quad \lambda_2(t)=\cdot\cdot\cdot=\lambda_N:=\frac{B'(t)}{t}.
\end{equation}

and, for any $t\in \R$, 
\begin{equation}
    F(t):=\int_0^t f(s)ds
\end{equation}
and\footnote{Under assumptions $(h_B)$, we note that $B'(t), B''(t)t\in C([0,+\infty))$. Therefore, the quantities $\Lambda_k$ are well-defined.} 
\begin{equation}\label{definizioni}
    \Lambda_k(t):=\int_0^t\lambda_k(|s|)sds.
\end{equation}

We remark that solutions of \eqref{formulazionedebole} are critical points of the functional 
\begin{equation}
    \int \Lambda_2(H(\nabla u))-F(u) dx.
\end{equation}

We then discuss the notion of stability for weak solutions to \eqref{equazione}. This is a rather delicate question, because our framework is very general. We address this in the following definition. 

\begin{defn}\label{stabilità2}
Given a domain $\Omega\subset \R^N$, we say that a weak solution $u\in C^{1}(\Omega)$ is stable if
\begin{equation}\label{stabilità}
\begin{split}
    L_u(\varphi,\varphi)&= \int_\Omega \langle A(\nabla u)\nabla \varphi,\nabla \varphi\rangle-f'(u)\varphi ^2 \\ &=\int_\Omega B''(H(\nabla u)) \langle \nabla H(\nabla
u), \nabla \varphi \rangle^2 \\&+
\int_{\Omega}B'(H(\nabla
u))\langle D^2H(\nabla u)\nabla \varphi,\nabla \varphi \rangle
-f'(u)  \varphi^2 \ge 0,
\end{split}
\end{equation}
for every $\varphi\in C^1_c(\Omega)$.
\end{defn}

\begin{rem}
	Note that the left hand side in Definition \ref{stabilità2} might be possibly infinity and the inequality is understood in this meaning. Furthermore, the anisotropic nature of our operator, causes that some care is needed to define the bilinear form on the critical set, we refer to \cite{CoFaVa2,CoFaVa1}. Computations are simpler in any case 
 and well-defined when  the support of the test function $\varphi$ is far from $Z_u$ ($\varphi\in C^1_c(\Omega\setminus \{\nabla u=0\})$) and  this is what we shall really exploit in our proofs. 	
\end{rem}

In this paper we will be interested in rigidity properties, such as Liouville-type theorems and one-dimensionality results. 

For deep developments in  the study of Liouville-type theorems for  stable solutions, possibly unbounded and 
sign-changing, we refer to \cite{DFSV,DF,Farina1,Farina2,FWM}.

Our first result is the following: 

\begin{thm}\label{DegiorgiN=2}
    Let $N=2$. Let $u\in C^{1}(\R^2)$ be a weak stable solution of \eqref{equazione}, with $\nabla u\in L^{\infty}(\R^2).$  Then there exist $u_0 :\R\rightarrow \R$ and $w \in S^{1}$ such that $u(x)=u_0(w\cdot x)$ for any $x\in \R^2$.
\end{thm}

\begin{rem}
  In the case $1<p<2$, the result is new even in the Euclidean case, since, unlike in \cite{FSV}, we did not assume $\{\nabla u=0\}=\emptyset$.  
\end{rem}

The proof that we present is based on a Poincaré type inequality and it is a variation given in \cite {FSV,SZ1,SZ2}. In this regard, we will extend a formula by Stemberg and Zumbrun which relates the second derivatives of a weak solution of the equation \eqref{equazione} with the principal curvatures of the corresponding level set, in the anisotropic setting.

The second symmetry result is the following

\begin{thm}\label{DegiorgiN=3}
    Let $N=3$. Let $u\in C^{1}(\R^3)\cap W^{1,\infty}(\R^3)$ be a weak solution of \eqref{equazione}.
    Suppose that 
    \begin{equation}\label{monotonia}
        \partial _{x_3}u >0.
    \end{equation}
    Then $u$ has one dimensional symmetry, namely there exist $u_0 :\R\rightarrow \R$ and $w \in S^{2}$ such that $u(x)=u_0(w\cdot x)$ for any $x\in \R^3$.
\end{thm}

We remark that we will prove Theorem \ref{DegiorgiN=3} under weaker assumptions, namely: $ \partial _{x_3}u \ge 0$, $\{\nabla u=0\}=\emptyset$. Furthermore, the condition \eqref{monotonia} implies the stability condition in \eqref{stabilità} (see Lemma \ref{monotoniaimplicastabilità} in Section \ref{dimostrazioneteorema}).

In the Euclidean framework and in the case of $B(t)=t^2/2$ and $f(t)=t-t^3$, Theorems \ref{DegiorgiN=2} and \ref{DegiorgiN=3} reduce to a problem proposed by De Giorgi, \cite{DEG}.
The first result was obtained for $N=2$ in \cite{BCN,GG}, and, subsequently, the case $N = 3$ was proved in \cite{AAC,AC}. Moreover, assuming an additional limit requirement on the solutions the De Giorgi’s conjecture is true for $4\le N\le 8$, (see \cite{Savin}). The conjecture fails for $N\ge 9,$ \cite{PKW}. 
In the case of $p-$Laplace operator and non-uniformly elliptic operators, one dimensional symmetry for solutions has also been dealt with in \cite{CGS,F1,FV,FSV,FSV2,VSS}. We also mention the recent result in \cite{figalli}. 

Concluding the introduction, let us mention that anisotropic operators arise in Finsler geometry, see \cite{BePa, CianSal, LePa, Sh1, Sh2}, as modelization e.g.~of material science \cite{CaHo, Gu}, relativity \cite{As}, biology \cite{AnInMa},
image processing \cite{EsOs, PeMa}.

\section{Notation and preliminary results} \label{Sec1}


The aim of this section is to introduce some notation and to recall technical results  about anisotropic elliptic operator that will be involved in the proof of the main results. For $a, b \in \R^N$ we denote by $a \otimes b$ the matrix whose entries are $(a \otimes b)_{ij}=a_ib_j$. We remark that for any $v,w \in \R^N$ it holds that:
$$\langle a \otimes b \ v, w \rangle = \langle b, v \rangle \langle a , w \rangle.$$

Let $H$ be a function belonging to $C^{3,\beta}_{loc}(\R^N \setminus \{0\})$.
$H$ is said a ``\emph{Finsler norm}'' if it satisfies the following assumptions:
\begin{itemize}
	\item[$(h_H)$]
\begin{enumerate}
    \item[(i)] $H(\xi)>0 \quad \forall \xi \in \R^N \setminus \{0\}$;

    \item[(ii)] $H(s \xi) = |s| H(\xi) \quad \forall \xi \in \R^N \setminus \{0\}, \, \forall s \in \R$;

    \item[(iii)] $H$ is \emph{uniformly elliptic}, that means the set $\mathcal{B}_1^H:=\{\xi \in \R^N  :  H(\xi) < 1\}$ is \emph{uniformly convex}
    \begin{equation}\label{uniformeellitticitadiH}
        \exists \Lambda > 0: \quad \langle \D^2H(\xi)v, v \rangle \geq \Lambda |v|^2 \quad \forall \xi \in \partial \mathcal{B}_1^H, \; \forall v \in \nabla H(\xi)^\bot.
    \end{equation}
\end{enumerate}
\end{itemize}

A set is said uniformly convex if the principal curvatures of its boundary are all strictly positive.

Since $H$ is a norm in $\R^N$, we immediately get that there exists $c_1,c_2>0$ such that:
\begin{equation}\label{H equiv euclidea}
c_1|\xi|\leq H(\xi)\leq c_2|\xi|,\,\quad\forall\, \xi\in\R^N.
\end{equation}

We recall that, since $H$ is a differentiable and $1$-homogeneous function, it holds the Euler's characterization result, i.e.
\begin{equation}\label{eulero}
\la\n H(\xi),\xi\ra=H(\xi) \qquad\forall\,\xi\in\R^N\setminus\{0\},
\end{equation}
and
\begin{equation}\label{grad 0 omog}
\nabla H(t\xi)=\hbox{sign}(t)\nabla H(\xi) \qquad\forall\,\xi\in\R^N\setminus\{0\}, \forall t\in\R\setminus \{0\}.
\end{equation}

Moreover
\begin{equation}\label{eq:PropFinsler1}
D^2H(t\xi)=\frac{1}{|t|}D^2H(\xi)\qquad \forall\,\xi\in\R^N\setminus\{0\}, \forall t\in \R\setminus\{0\}
\end{equation}
and
\begin{equation}\label{eq:PropFinsler2}
    D^2H(\xi)\xi=0 \qquad \forall\,\xi\in\R^N\setminus\{0\}.
\end{equation}
For more details on Finsler geometry see for instance \cite{BaChSh, BePa, JaSa}.

We refer to \cite{CoFaVa2, CoFaVa1, FeKa} for a discussion on the geometry of the Riemannian case and for interesting related results.

Recalling \eqref{matrice}, we conclude this section with a result that will be useful to us in the sequel (see \cite[Theorem $1.5$]{CoFaVa1}).

\begin{lem}[\cite{CoFaVa1}]\label{matricedefinitapositiva}
   For any $\xi \in \R^N\setminus\{0\}$ the following matrix
   \begin{equation}
       A(\xi)=B''(H(\xi))\nabla H(\xi)\otimes \nabla H(\xi)+B'(H(\xi))D^2H(\xi)
   \end{equation}
   is positive definite. In particular, there exist a constant $C$ such that 
   \begin{equation}
       A_{ij}(\xi)\eta_i\eta_j\ge C(\kappa+|\xi|)^{p-2}|\eta|^2,\quad \forall\xi\in \R^N\setminus\{0\},\quad \forall\eta\in \R^N.
   \end{equation}
  \end{lem}

\section{Proof of Theorem \ref{DegiorgiN=2}}
 The aim of this section is to prove Theorem \eqref{DegiorgiN=2}.


We start with 

\begin{lem}\label{Lemma1}
    Let $\Omega$ be a domain in $\R^N$. Let $u\in C^{1}(\R^N)$ be a weak solution of \eqref{equazione}.  Then, $u_i:=\frac{\partial u}{\partial x_i}$ is a solution of 
    \begin{equation}\label{u_irisolvelinearizzato}
        \begin{split}
        &\int_{\Omega\setminus Z_u} B''(H(\nabla u))\langle \nabla H(\nabla u),\nabla u_i\rangle\langle \nabla H(\nabla u),\nabla \varphi\rangle  \\ &+ B'(H(\nabla u)) \langle D^2H(\nabla u)\nabla u_i,\nabla \varphi\rangle =\int_{\Omega} f'(u)u_i\varphi\quad \forall \varphi\in C^1_c(\Omega),
    \end{split}
    \end{equation}
    for any $i=1,...,N.$

    Moreover we have 
    \begin{equation}\label{equazione22}
\begin{split}
&\int_{\Omega\cap \{\nabla u\neq 0\}} \sum_{i=1}^NB''(H(\nabla u))\langle\nabla H(\nabla u),\nabla u_i \rangle \langle \nabla H(\nabla u),\nabla(u_i\varphi^2)\rangle \\&+\sum_{i=1}^NB'(H(\nabla u)) \langle D^2H(\nabla u) \nabla u_i,\nabla (u_i\varphi^2)\rangle\,dx  =\int_{\Omega\cap \{\nabla u\neq 0\}} f'(u)|\nabla u|^2 \varphi ^2\,dx \quad \forall\varphi\in C^1_c(\Omega).
\end{split}
\end{equation}

\end{lem}

\begin{proof}
    We recall that $B'(H(\nabla u))\nabla H(\nabla u)\in W^{1,2}_{loc}(\Omega)$ (see \cite{Alberto2,CaRiSc,MS}). Using Stampacchia's theorem (see \cite[Theorem $6.19$]{S}) we deduce that the distributional derivative of $B'(H(\nabla u))\nabla H(\nabla u)$ is given by 
    \begin{equation}\label{distribuzionalederivata}
    \partial_i \left(B'(H(\nabla u))\nabla H(\nabla u)\right)=   \left\{\begin{array}{llll}
				 B''(H(\nabla u))\langle \nabla H(\nabla u),\nabla u_i\rangle \nabla H(\nabla u) \\ \\ + B'(H(\nabla u))  D^2H(\nabla u)\nabla u_i\qquad &\text{if }x\in \{\nabla u\neq0\}.\\ \\
				0 \qquad &\text{if } x\in \{\nabla u=0\} ,\\  
				
			\end{array}\right.
   \end{equation}
   for any $i=1,...,N.$
 Taking $\psi:=\partial_i \varphi$ in \eqref{formulazionedebole}, by \eqref{distribuzionalederivata} and integrating by parts, we get the \eqref{u_irisolvelinearizzato}.

Now we prove \eqref{equazione22}. Now, fix  $\epsilon>0$ and, for $\varphi\in C^1_c(\Omega)$, 
let us consider
\begin{equation}\label{test1}
\psi:=G_\varepsilon(u_i)\varphi^2
\end{equation}
 with $G_{\epsilon}\in C^{\infty}(\R)$ such that $G_{\epsilon}(t)=t$ if $|t|\ge 2\epsilon$, $G_{\epsilon}(t)=0$ if $|t|\le \epsilon$ and $G'_{\epsilon}(t)\le 3$ for any $t\in \R.$ Using $G_\varepsilon(u_i)\varphi^2\in C^1_c(\Omega)$ in  \eqref{u_irisolvelinearizzato} we have

\begin{equation}\label{zicu}
\begin{split}
    0= &\int_{\Omega}  B''(H(\nabla u))\langle \nabla H(\nabla u),\nabla u_i\rangle\langle \nabla H(\nabla u),\nabla (G_{\epsilon}(u_i)\varphi^2)\rangle \\&+  B'(H(\nabla u)) \langle D^2H(\nabla u)\nabla u_i,\nabla (G_{\epsilon}(u_i)\varphi^2)\rangle - f'(u)u_iG_{\epsilon}(u_i)\varphi^2 \,dx \\  =&\int_{\Omega} B''(H(\nabla u))\langle \nabla H(\nabla u),\nabla u_i\rangle\langle \nabla H(\nabla u),\nabla u_i\rangle G'_\varepsilon(u_i)\varphi^2 \\ & +B''(H(\nabla u))\langle \nabla H(\nabla u),\nabla u_i\rangle\langle \nabla H(\nabla u),\nabla (\varphi ^2)\rangle G_{\epsilon}(u_i)\\ &+B'(H(\nabla u)) \langle D^2H(\nabla u)\nabla u_i,\nabla u_i\rangle G'_\varepsilon(u_i)\varphi^2\\ &+B'(H(\nabla u)) \langle D^2H(\nabla u)\nabla u_i,\nabla (\varphi^2)\rangle G_{\epsilon}(u_i)- f'(u)u_iG_{\epsilon}(u_i)\varphi ^2 \,dx.
\end{split}
\end{equation}

Since that $\nabla H$ is a continuous and $0-$homogeneous function, we infer there exist $M>0$ such that
\begin{equation}\label{grad00}
    |\nabla H(\xi)|\le M \qquad \forall\xi\in \R^N\setminus \{0\}.
\end{equation}
By \eqref{grad00}, and since $B'(H(\nabla u))\nabla H(\nabla u)\in W^{1,2}_{loc}(\Omega)$, we note that

\begin{equation}\label{fact11}
\begin{split}
&B''(H(\nabla u))\langle \nabla H(\nabla u),\nabla u_i\rangle\langle \nabla H(\nabla u),\nabla u_i\rangle G'_\varepsilon(u_i)\varphi^2 \\ & \leq 3M^2 B''(H(\nabla u))|\nabla u_i|^2\varphi ^2\in L^1(\Omega).
\end{split}
\end{equation}
The fact that the right-hand side of \eqref{fact11} belongs to $L^1(\Omega)$ follows from \cite[Proposition 3.3]{CaRiSc} or \cite[Theorem 1.3]{BMV} (see also \cite{Alberto}). 

Recalling that $D^2H$ is a continuous and $-1-$homogeneous function, there exist a positive constant $M'$ such that $D^2H(\xi)\le \frac{M'}{H(\xi)}$ for all $\xi\in \R^N\setminus \{0\}.$ Similarly, we get 
\begin{equation}\label{fact2}
\begin{split}
    & B'(H(\nabla u)) \langle D^2H(\nabla u)\nabla u_i,\nabla u_i\rangle G'_\epsilon(u_i)\varphi^2 \\
    &\le 3M'\frac{B'(H(\nabla u))}{H(\nabla u)}|\nabla u_i|^2\varphi ^2 \in L^1(\Omega). 
    \end{split}
\end{equation}
As noted above, the right-hand side of \eqref{fact2} belongs to $L^1(\Omega)$.
Consequently, by \eqref{fact11} and \eqref{fact2}, passing $\epsilon \rightarrow 0$ in \eqref{zicu}, using dominated convergence theorem and summing over $i$ we get
\begin{equation}\label{equazione2}
\begin{split}
&\int_{\Omega\cap \{\nabla u\neq 0\}} \sum_{i=1}^NB''(H(\nabla u))\langle\nabla H(\nabla u),\nabla u_i \rangle \langle \nabla H(\nabla u),\nabla(u_i\varphi^2)\rangle \\&+\sum_{i=1}^NB'(H(\nabla u)) \langle D^2H(\nabla u) \nabla u_i,\nabla (u_i\varphi^2)\rangle\,dx  =\int_{\Omega\cap \{\nabla u\neq 0\}} f'(u)|\nabla u|^2 \varphi ^2\,dx.
\end{split}
\end{equation}

\end{proof}

To simplify future computations, we now state a key result that will be useful later.

\begin{lem}\label{Lemma2}
    Let $\Omega$ be a domain in $\R^N$. Let $u\in C^{1}(\R^N)$ be a stable weak solution of \eqref{equazione}.  Then, we have 
    
\begin{equation}\label{equazione144}
    \begin{split}
     &\int_{\Omega\cap \{\nabla u\neq 0\}} B''(H(\nabla u)) \langle \nabla H(\nabla u), \nabla |\nabla u|\rangle^2\varphi^2+B''(H(\nabla u)) \langle \nabla H(\nabla u), \nabla \varphi\rangle^2|\nabla u|^2 \\&+B''(H(\nabla u)) \langle \nabla H(\nabla u), \nabla |\nabla u|\rangle\langle \nabla H(\nabla u), \nabla \varphi\rangle2|\nabla u|\varphi\\&+B'(H(\nabla u))\langle D^2H(\nabla u)\nabla |\nabla u|,\nabla |\nabla u| \rangle\varphi ^2+B' (H(\nabla u))\langle D^2H(\nabla u)\nabla \varphi,\nabla \varphi \rangle |\nabla u|^2 \\ &+ B'(H(\nabla u)) \langle D^2H(\nabla u) \nabla |\nabla u|,\nabla \varphi\rangle |\nabla u|\varphi+B'(H(\nabla u)) \langle D^2H(\nabla u) \nabla \varphi,\nabla |\nabla u|\rangle |\nabla u|\varphi \\ &-f'(u)|\nabla u|^2\varphi ^2\,dx\ge 0\qquad \forall\varphi\in C^1_c(\Omega).
    \end{split}
\end{equation}
    
\end{lem}

\begin{proof}
    For $\epsilon>0$, we consider \begin{equation}\label{test2}
        \psi:= G_{\epsilon}(|\nabla u|)\varphi
    \end{equation} 
    where $\varphi\in C^1_c(\Omega)$ and $G_{\epsilon}\in C^{\infty}(\R)$ such that $G_{\epsilon}(t)=t$ if $|t|\ge 2\epsilon$, $G_{\epsilon}(t)=0$ if $|t|\le \epsilon$ and $G'_{\epsilon}(t)\le 3$ for any $t\in \R.$ 

We observe that $G_{\epsilon}(|\nabla u|)\varphi\in C^1_c(\Omega\setminus Z_u)$  is a good test function in \eqref{stabilità}. Substituting \eqref{test2} in \eqref{stabilità} we get 

\begin{equation}\label{eq1}
\begin{split}
&\int_{\Omega}  B''(H(\nabla u))\langle\nabla H(\nabla u),\nabla \left(G_{\epsilon}(|\nabla u|)\varphi\right) \rangle^2 \\ &+ B'(H(\nabla u))\langle D^2H(\nabla u)\nabla \left(G_{\epsilon}(|\nabla u|)\varphi\right), \nabla \left(G_{\epsilon}(|\nabla u|)\varphi\right)\rangle -f'(u) G_{\epsilon}(|\nabla u|)^2\varphi ^2 \,dx\ge 0.
\end{split}
\end{equation}

Since $\nabla \psi= G'_{\epsilon}(|\nabla u|)\nabla |\nabla u|\varphi+G_{\epsilon}(|\nabla u|)\nabla \varphi$, \eqref{eq1} becomes

\begin{equation}\label{eq2}
    \begin{split}
     0\le &\int_\Omega B''(H(\nabla u)) \langle \nabla H(\nabla u), \nabla |\nabla u|\rangle^2G_{\epsilon}'(|\nabla u |)^2\varphi^2+B''(H(\nabla u)) \langle \nabla H(\nabla u), \nabla \varphi\rangle^2G_{\epsilon}(|\nabla u|)^2 \\&+B''(H(\nabla u)) \langle \nabla H(\nabla u), \nabla |\nabla u|\rangle\langle \nabla H(\nabla u), \nabla \varphi\rangle2G_{\epsilon}(|\nabla u|)G_{\epsilon}'(|\nabla u|)\varphi\\&+B' (H(\nabla u))\langle D^2H(\nabla u)\nabla |\nabla u|,\nabla |\nabla u| \rangle G_{\epsilon}'(|\nabla u|)^2\varphi ^2\\&+B' (H(\nabla u))\langle D^2H(\nabla u)\nabla \varphi,\nabla \varphi \rangle G_{\epsilon}(|\nabla u|)^2 \\ &+ B'(H(\nabla u)) \langle D^2H(\nabla u) \nabla |\nabla u|,\nabla \varphi\rangle G_{\epsilon}(|\nabla u|)G_{\epsilon}'(|\nabla u|)\varphi \\&+B'(H(\nabla u)) \langle D^2H(\nabla u) \nabla \varphi,\nabla |\nabla u|\rangle G_{\epsilon}(|\nabla u|)G_{\epsilon}'(|\nabla u|)\varphi -f'(u)G_{\epsilon}(|\nabla u|)^2\varphi^2\,dx.
    \end{split}
\end{equation}

We note that, for any $m>0$ and $\Psi\in L^1(\Omega)$, we have 

\begin{equation}\label{fact1}
\lim_{\epsilon \rightarrow 0}\int_\Omega G'_{\epsilon}(|\nabla u|)^m\Psi\,dx-\int_{\Omega \cap \{\nabla u\neq 0\}}\Psi\,dx \le \lim_{\epsilon \rightarrow 0}\int_{\Omega \cap \{\nabla u\neq 0\}} 3^m|\Psi| \chi _{\epsilon \le |\nabla u|\le 2\epsilon}\,dx=0.
\end{equation}

We now let $\epsilon \to 0$ in \eqref{eq2}, which is allowed thanks to the regularity results in \cite[Proposition 3.3]{CaRiSc} or \cite[Theorem 1.3]{BMV}. Using \eqref{fact1} and the dominated convergence theorem, proceeding as in the proof of the Lemma \ref{Lemma1} we conclude that

\begin{equation}\label{equazione1}
    \begin{split}
     0\le &\int_{\Omega\cap\{\nabla u\neq 0\}} B''(H(\nabla u)) \langle \nabla H(\nabla u), \nabla |\nabla u|\rangle^2\varphi^2+B''(H(\nabla u)) \langle \nabla H(\nabla u), \nabla \varphi\rangle^2|\nabla u|^2 \\&+B''(H(\nabla u)) \langle \nabla H(\nabla u), \nabla |\nabla u|\rangle\langle \nabla H(\nabla u), \nabla \varphi\rangle2|\nabla u|\varphi\\&+B'(H(\nabla u))\langle D^2H(\nabla u)\nabla |\nabla u|,\nabla |\nabla u| \rangle\varphi ^2+B' (H(\nabla u))\langle D^2H(\nabla u)\nabla \varphi,\nabla \varphi \rangle |\nabla u|^2 \\ &+ B'(H(\nabla u)) \langle D^2H(\nabla u) \nabla |\nabla u|,\nabla \varphi\rangle |\nabla u|\varphi+B'(H(\nabla u)) \langle D^2H(\nabla u) \nabla \varphi,\nabla |\nabla u|\rangle |\nabla u|\varphi \\ &-f'(u)|\nabla u|^2\varphi ^2\,dx.
    \end{split}
\end{equation}

\end{proof}

Now we recall the definition of the tangential gradient with respect to a regular level set. For $u\in C^1(\R^N)$, we define the level set of $u$ at $x$ as 
\begin{equation}
    L_{u,x}:=\{y\in \R^N \text{ such that } u(y)=u(x)\}.
\end{equation}

Given $f\in C^1(B_r(x))$, for $r>0$, the tangential gradient denotes the orthogonal projection of the gradient along this level set, explicitly 

\begin{equation}\label{tangenzialegradiente}
\nabla_{L_{u,x}}f(x):=\nabla f(x)-\langle\nabla f(x),\frac{\nabla u(x)}{|\nabla u(x)|}\rangle\frac{\nabla u(x)}{|\nabla u(x)|}.
\end{equation}

In the sequel we give an extension, in the anisotropic framework, of a formula obtained in \cite{SZ1,SZ2}. 

\begin{lem}\label{Zumbrun}
    Let $\Omega\subset\R^N$ be an open set. Then for any function $u\in C^1(\Omega)\cap C^2(\{\nabla u\neq 0\})$ and any $x\in \{\nabla u\neq 0\}$, we have 
    \begin{equation}\label{zumpo}
    \begin{split}
   \sum_{i=1}^N \langle D^2H(\nabla u)\nabla u_i,\nabla u_i\rangle-\langle D^2H(\nabla u)\nabla |\nabla u|,\nabla |\nabla u|\rangle \\  \ge\frac{C_H}{H(\nabla u)}|\nabla u|^2 \sum_{i=1}^{N-1}k_{i,u}^2,
    \end{split}
\end{equation}

where $k_{i,u}(x)$, for $i=1,...,N-1$, denote the principal curvature of the level set $L_{u,x}$ of $u$ at $x$.
\end{lem}

\begin{proof}
We suppose $\nabla u(x_0) \neq 0$, for some $x_0\in \Omega.$ For all $x$ in a neighbourhood of $x_0$, we set $\tau_n(x):=\frac{\nabla u(x)}{|\nabla u(x)|}$ and we denote with $\{\tau_i\}_{i=1,...,N-1}$ an orthonormal basis for the tangent plane of the level set $\{x : u(x) = u(x_0)\}$ at $x_0$. 
    We note that 
    \begin{equation}\label{identita}
        \nabla u_i=\sum_{j=1}^N \langle \nabla u_i, \tau_j\rangle \tau_j, 
    \end{equation}
    for $i=1,...,N.$
By \eqref{identita}, we get 
\begin{equation}
\begin{split}
   &\sum_{i=1}^N \langle D^2H(\nabla u)\nabla u_i,\nabla u_i\rangle-\langle D^2H(\nabla u)\nabla |\nabla u|,\nabla |\nabla u|\rangle\\ &=\sum _{i,j,l=1}^N  \langle D^2H(\nabla u) \tau_j,\tau_l\rangle \langle\nabla u_i,\tau_j\rangle \langle\nabla u_i,\tau_l\rangle-\langle D^2H(\nabla u)\nabla |\nabla u|,\nabla |\nabla u|\rangle.
    \end{split}
\end{equation}

Since $D^2H(\xi)\xi=0$ for every $\xi\in \R^N\setminus \{0\},$ we have 
\begin{equation}
\begin{split}
   &\sum_{i=1}^N \langle D^2H(\nabla u)\nabla u_i,\nabla u_i\rangle-\langle D^2H(\nabla u)\nabla |\nabla u|,\nabla |\nabla u|\rangle\\ &=\sum _{i=1}^N\sum_{j,l=1}^{N-1}  \langle D^2H(\nabla u) \tau_j,\tau_l\rangle \langle\nabla u_i,\tau_j\rangle \langle\nabla u_i,\tau_l\rangle-\langle D^2H(\nabla u)\nabla |\nabla u|,\nabla |\nabla u|\rangle.
    \end{split}
\end{equation}

Since $\tau_n$ is orthogonal to $\tau _j$, for $j=1,...,N-1$, we deduce that 
\begin{equation}\label{identi}
    \langle \nabla u_i,\tau_j\rangle =\langle \partial_i\left(|\nabla u|\tau_n\right),\tau_j\rangle=|\nabla u|\langle (\tau_n)_{x_i},\tau_j\rangle,
\end{equation}

for $i=1,...,N$ and $j=1,...,N-1$.

From previous identity we get
\begin{equation}\label{equa}
\begin{split}
   &\sum_{i=1}^N \langle D^2H(\nabla u)\nabla u_i,\nabla u_i\rangle-\langle D^2H(\nabla u)\nabla |\nabla u|,\nabla |\nabla u|\rangle\\ &=\sum _{i=1}^{N-1}\sum_{j,l=1}^{N-1}  |\nabla u|^2\langle D^2H(\nabla u) \tau_j,\tau_l\rangle \langle (\tau_n)_{x_i},\tau_j\rangle\langle (\tau_n)_{x_i},\tau_l\rangle\\&+\sum_{j,l=1}^{N-1}  \langle D^2H(\nabla u) \tau_j,\tau_l\rangle \langle\nabla u_N,\tau_j\rangle \langle\nabla u_N,\tau_l\rangle-\langle D^2H(\nabla u)\nabla |\nabla u|,\nabla |\nabla u|\rangle.
    \end{split}
\end{equation}

Without loss of generality, we now assume that $\tau_i(x_0)$ points in the direction of the coordinate $x_i$ for $i = 1,...,N$.

At the point $x=x_0$, we note that 
\begin{equation}\label{secs}
\begin{split}
    \nabla |\nabla u|(x_0)&=\left(\sum_{p=1}^N \langle\nabla |\nabla u|,\tau_p \rangle \tau_p\right)(x_0)\\&=\left(\sum_{p=1}^N\partial_p(|\nabla u|)\tau_p\right)(x_0)=\left(\sum_{p=1}^N \langle\nabla u_p,\tau_n \rangle \tau_p\right)(x_0).
    \end{split}
\end{equation}

By \eqref{secs}, since $D^2H$ is a symmetric matrix and $D^2H(\xi)\xi=0$, $\forall\xi \in \R^{N}\setminus \{0\}$, we have 

\begin{equation}\label{equa2}
    \begin{split}
        \langle D^2H(\nabla u)\nabla |\nabla u|,\nabla |\nabla u|\rangle(x_0) &= \left(\sum_{p,s=1}^{N-1} \langle D^2H(\nabla u) \tau_p,\tau_s \rangle\langle \nabla u_p,\tau_n\rangle\langle\nabla u_s,\tau_n\rangle\right)(x_0)\\ &=\left(\sum_{p,s=1}^{N-1}\langle D^2H(\nabla u) \tau_p,\tau_s\rangle \partial_N(u_p)\partial_N(u_s)\right)(x_0) \\ &=\left(\sum_{p,s=1}^{N-1} \langle D^2H(\nabla u) \tau_p,\tau_s\rangle \langle\nabla u_N,\tau_p\rangle \langle\nabla u_N,\tau_s\rangle\right)(x_0).
    \end{split}
\end{equation}

By \eqref{equa} and \eqref{equa2}, we have at $x_0$ 

\begin{equation}
    \begin{split}
       &\sum_{i=1}^N \langle D^2H(\nabla u)\nabla u_i,\nabla u_i\rangle-\langle D^2H(\nabla u)\nabla |\nabla u|,\nabla |\nabla u|\rangle\\&=\sum _{i=1}^{N-1}\sum_{j,l=1}^{N-1}  |\nabla u|^2\langle D^2H(\nabla u) \tau_j,\tau_l\rangle \langle (\tau_n)_{x_i},\tau_j\rangle\langle (\tau_n)_{x_i},\tau_l\rangle\\ &=\sum _{i=1}^{N-1}  |\nabla u|^2\langle D^2H(\nabla u) \sum_{j=1}^{N-1}B_{ij,u}\tau_j,\sum_{l=1}^{N-1}B_{il,u}\tau_l\rangle 
    \end{split}
\end{equation}
where $B=\left(B_{ij,u}\right)$ denotes the second fundamental form at $x_0$ associated with the level set $\{x : u(x) = u(x_0)\}$.

Since $H$ is uniformly elliptic, see \eqref{uniformeellitticitadiH} and $D^2H$ is $-1$-homogeneous function, there exist a constant $C_H>0$, depending on $H$, such that
\begin{equation}
\begin{split}
    &\sum _{i=1}^{N-1}  |\nabla u|^2\langle D^2H(\nabla u) \sum_{j=1}^{N-1}B_{ij,u}\tau_j,\sum_{l=1}^{N-1}B_{il,u}\tau_l\rangle \\ &\ge \frac{C_H}{H(\nabla u)} |\nabla u|^2\sum_{i=1}^{N-1}\left|\sum_{j=1}^{N-1}B_{ij,u}\tau_j\right|^2= \frac{C_H}{H(\nabla u)} |\nabla u|^2\sum_{i,j=1}^{N-1} B_{ij,u}^2 \\&=\frac{C_H}{H(\nabla u)}|\nabla u|^2 \sum_{i=1}^{N-1}k_{i,u}^2,
    \end{split}
\end{equation}

where $k_{i,u}(x)$, for $i=1,...,N-1$, denote the principal curvature of the level set $L_{u,x_0}$ of $u$ at $x_0$.
    
\end{proof}

We are now ready to prove a key result for the proof of the Theorem \ref{DegiorgiN=2}.

\begin{thm}\label{finirà}
    Let $\Omega\subset \R^N$ be open. Let $u\in C^{1}(\Omega)$ be a  weak stable solution of \eqref{equazione}.  For any $x\in \Omega\cap \{\nabla u\neq 0\}$ let $k_{i,u}(x)$, for $i=1,...,N-1$, denote the principal curvature of the level set $L_{u,x}$ of $u$ at $x$.
    Then there exist a positive constant $C_H$, depending on $H$, such that
    \begin{equation}\label{formulafondamentale}
       \begin{split}
       & \int_{\Omega\cap\{\nabla u\neq 0\}} \left(\lambda_2(H(\nabla u))  |\nabla u|^2 \sum_{i=1}^{N-1} k_{i,u}^2  +\lambda_1(H(\nabla u))\ \left|\nabla _{L_{u,x}}H(\nabla u)\right|^2\right)\varphi^2\,dx \\ &\le C_H\int_{\Omega} \langle A(\nabla u)\nabla \varphi,\nabla \varphi\rangle|\nabla u|^2 \,dx, 
    \end{split}
\end{equation}
    
    for any  $\varphi\in C^1_c(\Omega).$
\end{thm}

\begin{proof}

Using Lemma \ref{Lemma1} and Lemma \ref{Lemma2}, in particular by \eqref{equazione22} and \eqref{equazione144} we deduce that 
\begin{equation}
    \begin{split}
     0\le &\int_{\Omega\cap\{\nabla u\neq 0\}} B''(H(\nabla u)) \langle \nabla H(\nabla u), \nabla |\nabla u|\rangle^2\varphi^2+B''(H(\nabla u)) \langle \nabla H(\nabla u), \nabla \varphi\rangle^2|\nabla u|^2 \\&+B''(H(\nabla u)) \langle \nabla H(\nabla u), \nabla |\nabla u|\rangle\langle \nabla H(\nabla u), \nabla \varphi\rangle2|\nabla u|\varphi\\&+B' (H(\nabla u))\langle D^2H(\nabla u)\nabla |\nabla u|,\nabla |\nabla u| \rangle\varphi ^2+B' (H(\nabla u))\langle D^2H(\nabla u)\nabla \varphi,\nabla \varphi \rangle |\nabla u|^2 \\ &+ B'(H(\nabla u)) \langle D^2H(\nabla u) \nabla |\nabla u|,\nabla \varphi\rangle |\nabla u|\varphi+B'(H(\nabla u)) \langle D^2H(\nabla u) \nabla \varphi,\nabla |\nabla u|\rangle |\nabla u|\varphi \\ &-\sum_{i=1}^NB''(H(\nabla u))\langle\nabla H(\nabla u),\nabla u_i \rangle \langle \nabla H(\nabla u),\nabla(u_i\varphi^2)\rangle \\&- \sum_{i=1}^N B'(H(\nabla u)) \langle D^2H(\nabla u) \nabla u_i,\nabla (u_i\varphi^2)\rangle\,dx.
    \end{split}
\end{equation}

Since $\nabla (u_i\varphi^2)= \nabla u_i \varphi ^2+2\varphi u_i\nabla \varphi$, we have 

\begin{equation}
    \begin{split}
    0\le &\int_{\Omega\cap\{\nabla u\neq 0\}} \left(B'(H(\nabla u))\langle D^2H(\nabla u)\nabla |\nabla u|,\nabla |\nabla u|\rangle -\sum_{i=1}^{N}B'(H(\nabla u))\langle D^2H(\nabla u)\nabla u_i,\nabla u_i\rangle\right)\varphi^2 \\&+\left(B''(H(\nabla u)) \langle \nabla H(\nabla u), \nabla |\nabla u|\rangle^2-\sum_{i=1}^NB''(H(\nabla u))\langle \nabla H(\nabla u),\nabla u_i\rangle^2\right)\varphi^2 \\&+\left(B''(H(\nabla u))\langle \nabla H(\nabla u),\nabla \varphi\rangle^2 +B'(H(\nabla u))\langle D^2H(\nabla u)\nabla \varphi,\nabla \varphi\right)|\nabla u|^2\\&+B'(H(\nabla u)) \langle D^2H(\nabla u) \nabla |\nabla u|,\nabla \varphi\rangle |\nabla u|\varphi+B'(H(\nabla u)) \langle D^2H(\nabla u) \nabla \varphi,\nabla |\nabla u|\rangle |\nabla u|\varphi \\ &+ B''(H(\nabla u))\langle \nabla H(\nabla u), \nabla |\nabla u|\rangle\langle \nabla H(\nabla u), \nabla \varphi\rangle2|\nabla u|\varphi\\&-\sum_{i=1}^NB''(H(\nabla u))\langle\nabla H(\nabla u),\nabla u_i \rangle \langle \nabla H(\nabla u),\nabla\varphi\rangle 2u_i\varphi \\&-\sum_{i=1}^NB'(H(\nabla u)) \langle D^2H(\nabla u) \nabla u_i,\nabla \varphi\rangle 2u_i\varphi\,dx.
    \end{split}
\end{equation}

Since $D^2H$ is a symmetric matrix and $|\nabla u| \nabla |\nabla u|=\sum_{i=1}^N\nabla u_i u_i$, we get 
\begin{equation}\label{equazione3}
    \begin{split}
        0\le &\int_{\Omega\cap\{\nabla u\neq 0\}} \left(B'(H(\nabla u))\langle D^2H(\nabla u)\nabla |\nabla u|,\nabla |\nabla u|\rangle-\sum_{i=1}^NB'(H(\nabla u))\langle D^2H(\nabla u)\nabla u_i,\nabla u_i\rangle\right)\varphi^2 \\&+\left(B''(H(\nabla u)) \langle \nabla H(\nabla u), \nabla |\nabla u|\rangle^2-\sum_{i=1}^NB''(H(\nabla u))\langle \nabla H(\nabla u),\nabla u_i\rangle^2\right)\varphi^2 \\&+\left(B''(H(\nabla u))\langle \nabla H(\nabla u),\nabla \varphi\rangle^2 +B' (H(\nabla u))\langle D^2H(\nabla u)\nabla \varphi,\nabla \varphi\rangle\right)|\nabla u|^2\,dx.
    \end{split}
\end{equation}

We note that 
\begin{equation}\label{normatangenziale}
\left|\nabla _{L_{u,x}}H(\nabla u)\right|^2=\sum_{i=1}^N\langle \nabla H(\nabla u),\nabla u_i\rangle^2-\langle \nabla H(\nabla u),\nabla |\nabla u|\rangle ^2.   
\end{equation}

Therefore \eqref{equazione3} becomes
\begin{equation}
    \begin{split}
       0\le &\int_{\Omega\cap\{\nabla u\neq 0\}} B'(H(\nabla u))\varphi^2\left(\langle D^2H(\nabla u)\nabla |\nabla u|,\nabla |\nabla u|\rangle-\sum_{i=1}^N\langle D^2H(\nabla u)\nabla u_i,\nabla u_i\rangle\right) \\ &-B''(H(\nabla u)) \left|\nabla _{L_{u,x}}H(\nabla u)\right|^2\varphi^2 \\&+\left(B''(H(\nabla u))\langle \nabla H(\nabla u),\nabla \varphi\rangle^2 +B' (H(\nabla u))\langle D^2H(\nabla u)\nabla \varphi,\nabla \varphi\rangle\right)|\nabla u|^2\,dx.
    \end{split}
\end{equation}

By Lemma \ref{Zumbrun}, we get

\begin{equation}\label{a1}
    \begin{split}
       & \int_{\Omega\cap\{\nabla u\neq 0\}} \frac{B'(H(\nabla u))}{H(\nabla u)}\varphi^2  |\nabla u|^2 \sum_{i=1}^{N-1} k_{i,u}^2  +B''(H(\nabla u)) \left|\nabla _{L_{u,x}}H(\nabla u)\right|^2\varphi^2\,dx \\ &\le C\int_{\Omega} \langle A(\nabla u)\nabla \varphi,\nabla\varphi\rangle\|\nabla u|^2,dx, 
    \end{split}
\end{equation}
where $k_{i,u}(x)$, for $i=1,...,N-1$, denote the principal curvature of the level set $L_{u,x}$ of $u$ at $x$ and $C_H$ is a positive constant depending on $H$.
We remark that the integral on the right-hand side of \eqref{a1} is well defined also on $\{x\in \Omega :\nabla u(x)=0\}. $

\end{proof}

Using the previous result we give 

\begin{proof}[Proof of Theorem \ref{DegiorgiN=2}]

If $\nabla u(x)=0$ for any $x\in \R^2$, the one-dimensional symmetry is trivial.

Since $|\nabla u|$ is bounded in $\R^2$ and $A(\nabla u)|\nabla u|^2 L^{\infty}_{loc}([0,+\infty))$ (by $(h_b)-(iii)$ and assumptions on $H$), by \eqref{formulafondamentale} we obtain 

\begin{equation}\label{prop2}
\begin{split}
      \int_{\{\nabla u\neq 0\}} &\frac{B'(H(\nabla u))}{H(\nabla u)}\varphi^2  |\nabla u|^2  k_{1,u}^2  +B''(H(\nabla u)) \left|\nabla _{L_{u,x}}H(\nabla u)\right|^2\varphi^2 \\ &\le C  \int_{\R^2} |\nabla \varphi |^2
      \end{split}
\end{equation}
for suitable $C>0$. 

We fix $R>0$, and let us consider the function 
\begin{equation}\label{cutfunction}
  \varphi:= \max \left \{0,\min \left\{1,\frac{\ln (R^2/|x|)}{\ln R}\right\}\right\}.
\end{equation}
We remark that $\varphi\in W^{1,\infty}(\Omega)$, compactly supported. Therefore by density argument we can use $\varphi$ in \eqref{prop2}.

Substituting \eqref{cutfunction} in \eqref{prop2} we get 
\begin{equation}\label{nullo1}
    k_1(x)=0
\end{equation}
and 
\begin{equation}\label{nullo2}
    \nabla _{L_{u,x}}H(\nabla u(x))=0
\end{equation}
for any $x\in \{\nabla u\neq 0\}.$

By \eqref{nullo1} and \eqref{nullo2}, every a non-empty connected component $\overline L$ of $L_{u,x}\cap \{\nabla u\neq 0\}$ is a flat hyperplane (for all details see \cite[Section $2.4$]{FSV}). Therefore $u$ is constant on these hyperplane, since each of them lies on a level set. On the contrary, $u$ is constant on any other possible hyperplane parallel to the ones of the above family, because the gradient is zero there. From this $u$ has one-dimensional symmetry. 

\end{proof}
\section{Preliminaries for the proof of Theorem \ref{DegiorgiN=3}}
In the first part of this section we aim to classify the solutions of \eqref{equazione}, in the case $N=1$. The second part is devoted to the study of solutions that are monotone in one direction.
\subsection{ODE analysis}

We start this section analysing the following ODE
\begin{equation}\label{ode}
    \left(B'(H(h'(t)))H'(h'(t))\right)'+f(h(t))=0
\end{equation}
for any $t\in \R$.

We consider weak solutions $h\in C^1(\R)\cap W^{1,\infty}(\R).$ Here, the function $H:\R\rightarrow \R$ is a Finsler norm, $B$ satisfies the assumptions $(h_B)$-$(iii)$ and $f$ is a continuous function.

\begin{rem}
    Since a Finsler norm $H:\R\rightarrow \R$ is given by $H(\xi)=c|\xi|,$ where $c$ is a positive constant, the ODE \eqref{ode} simplifies to the equation analyzed in \cite{FSV}. For the sake of completeness, we recall some results related to the ODE \eqref{ode} that will be used later. For the details of the proofs, we refer to \cite{FSV}.
\end{rem}

\begin{lem}\label{basta2}
   Let $h$ is a weak solution of \eqref{ode}, then we have 
    \begin{equation}\label{formulina}
        \Lambda_1(H(h'(t))+F(h(t))=F(\inf h)=F(\sup h)
    \end{equation}
     for any $t\in \R$. 
\end{lem}

We remark that, as a consequence of Lemma \ref{basta2}, 
\begin{equation}\label{remarkino}
 F(\inf h)=F(\sup h)=F(h(t^*)),   \end{equation}
for any $t^*\in \{h'=0\}.$

From previous result follows  the classification of the ODE in \eqref{ode}:

\begin{lem}\label{classificazioneode}
    One of the following holds:
    \begin{itemize}
        \item [(i)] $h$ is constant
        \item [(ii)] $\{h'=0\}=\emptyset$
        \item [(iii)] $h'(t)\neq 0$ for any $t$ in a bounded interval $(\beta_1,\beta_2)$ with $h'(\beta_1)=h'(\beta_2)=0$ and 
        \begin{equation}\label{abbaia2}
            F(h(\beta_1))=F(h(\beta_2))=F(\inf h)=F(\sup h).
        \end{equation}
        \item [(iv)] $h'(t)\neq 0$ for any $t$ in a unbounded interval either $(-\infty,\beta)$ or $(\beta,+\infty)$ with $h'(\beta)=0$ and 
        \begin{equation}\label{abbaia3}
           F(h(\beta))=F(\inf h)=F(\sup h).
        \end{equation}
    \end{itemize}
\end{lem}




\subsection{Behavior of the one-dimensional profiles at infinity}

Given $v: \R^N\rightarrow \R,$ we set $v^t(x):=v(x',x_N+t)$, for any $x=(x',x_N)\in \R^{N-1} \times \R$ and $t\in R$. We now prove a Poincarè type formula, in one dimension less.


\begin{thm}\label{analisiinfinito}
    Let $u\in C^{1,\alpha}_{loc}(\R^N)\cap W^{1,\infty}(\R^N)$ be a stable weak solution of \eqref{equazione}. 
    Suppose $\partial _{x_N} u\ge 0.$
    We set 
    \begin{equation}\label{ubassousopra}
    \begin{split}
        &\underline u(x'):=\lim_{t\rightarrow -\infty} u(x',x_N+t) \\
       & \overline u(x'):=\lim_{t\rightarrow +\infty} u(x',x_N+t) \\
       & \overline H (x')=\underline H(x'):=H(x',0) \\
        &\tilde A_{ij}(\xi):=A_{ij}(\xi) \qquad i,j=1,...,N-1.
        \end{split}
    \end{equation}

Then $\underline u,\overline u\in C^1(\R^{N-1})$ are weak solutions of \eqref{equazione} in $\R^{N-1}$. Moreover, we have

 \begin{equation}\label{formulafondamentale2}
       \begin{split}
       & \int_{\R^{N-1}\cap\{\nabla \overline u\neq 0\}} \lambda_2(\overline H(\nabla \overline u))\varphi^2  |\nabla \overline u|^2 \sum_{i=1}^{N-2} k_{i,\overline u}^2  +\lambda_1(\overline H(\nabla \overline u)) \left|\nabla _{L_{\overline u,x'}}\overline H(\nabla \overline u)\right|^2\varphi^2 \\ &\le C\int_{\R^{N-1}} \langle \tilde A(\nabla \overline u)\nabla\varphi,\nabla\varphi\rangle|\nabla \overline u|^2 
    \end{split}
\end{equation}
    for any  $\varphi\in C^1_c(\R^{N-1})$ and $C$ is a positive constant depending on $\overline{H}$ (and analogous claim holds for $\underline u$.)
    \end{thm}

\begin{proof}
   We note that $\underline u$ and $\overline u$ are well defined since $u$ is monotone in the direction $x_N$ and bounded.
   Since $u\in W^{1,\infty}(\R^N)$, by \cite[Proposition $3.1$]{CoFaVa2} and \cite[Theorem $1$]{T}, $$\|u_t\|_{C^{1,\alpha}(\R^N)}\le C,$$
   where $C$ is a positve constant not depending on $t$. From Ascoli-Arzela Theorem we deduce that $u^t$ tends to $\overline{u}$, for $t\rightarrow  +\infty$, in the norm $\|\cdot\|_{C^{1,\alpha '}_{loc}(\R^N)}$, with $\alpha '<\alpha.$ (same claim for $\underline u$.)

We now fix $\epsilon >0$ and we set $A_\epsilon:=\{(x',x_N)\in \R\times \R^{N-1}\text{ such that } |\nabla \overline u|\ge \epsilon\}$. We consider a compact set $K\subset\subset A_\varepsilon$. Since $u_t$ converges to $\overline{u}$ in the norm $C^1(K)$, we deduce that $|\nabla u(x',x_N+t)|\ge \varepsilon/2$, for $t$ sufficiently large.  Since $u\in C^{2,\alpha}_{loc}(\R^N\setminus \{\nabla u\neq 0\})$, by Shauder estimates we obtain 
$$\|u_t\|_{C^{2,\alpha}(K)}\le C,$$ where $C$ is a positive constant not depending on $t$.
From Ascoli theorem $u^t$ tends to $\overline{u}$, for $t\rightarrow  +\infty$, in the norm $\|\cdot\|_{C^{2,\alpha '}_{loc}(A_\epsilon)}$ (same claim for $\underline u$.)

Let us consider $\varphi:=\varphi_1(x')\varphi_2(x_N)$ with $\varphi_1\in C^1_c(\R^{N-1})$ and $\varphi\in C^1_c(\R)$ satisfying 
\begin{equation}\label{zicur}\int_\R \varphi_2(x_N) dx_N=1.\end{equation}

We note that $u_t$ is a weak solution of \eqref{equazione}, that is 
\begin{equation}\label{formulazionedebole33}
    \int_{\R^N}B'(H(\nabla  u_t))\langle \nabla H(\nabla u_t),\nabla \varphi \rangle \,dx=\int_{\R^N} f( u_t)\varphi\,dx
\end{equation}
Substituting $\varphi$ in \eqref{formulazionedebole33}, passing to the limit, we get 

\begin{equation}
    \int_{\R^N}B'(H(\nabla \overline u))\langle \nabla H(\nabla \overline{u}),\nabla \left(\varphi_1(x')\varphi_2(x_N)\right)\rangle \,dx'\,dx_N=\int_{\R^N} f(\overline u)\varphi_1(x')\varphi_2(x_N)\,dx'\,dx_N. 
\end{equation}

By \eqref{zicur} and recalling \eqref{ubassousopra} we have 

\begin{equation}\label{formulazionedebole21}
\begin{split}
    &\int_{\R^{N-1}} B'(\overline H(\nabla \overline u))\langle \nabla \overline H(\nabla \overline u),\nabla \varphi_1\rangle dx'\\&+ \int_{\R^N}B'(\overline H(\nabla \overline u(x')))\partial_{x_N}  H(\nabla \overline u(x')))\varphi_2'(x_N)\varphi_1(x')\,dx'\,dx_N =\int_{\R^{N-1}}f(\overline u)\varphi_1\,dx'.
    \end{split}
\end{equation}

Since $\varphi_2$ is compactly supported in $\R$, we deduce $$\int_\R \varphi_2'(s)\,ds=0.$$ By Fubini Theorem we get
\begin{equation}\label{formulazionedebole2}
    \int_{\R^{N-1}} B'(\overline H(\nabla \overline u))\langle \nabla \overline H(\nabla \overline u),\nabla \varphi_1\rangle \,dx' =\int_{\R^{N-1}}f(\overline u)\varphi_1\,dx',
\end{equation}
 for any $\varphi_1\in C^1_c(\R^{N-1})$, and the same for $\underline u$.

Since $ u(x',x_N+t)\rightarrow \overline u$, for $t\rightarrow +\infty$, in the norm $\|\cdot\|_{C^{2,\alpha'}_{loc}(A_\epsilon)}$, and recalling the definition of the tangential gradient \eqref{tangenzialegradiente} we deduce
\begin{equation}\label{tang}
    \begin{split}
        &\lim_{t\rightarrow +\infty} \nabla_{L_{u,(x',x_N+t)}}H(\nabla u)(x',x_N+t)\\&=\lim_{t\rightarrow +\infty} D^2u(x',x_N+t) \nabla H(\nabla u(x',x_N+t))\\&- \langle D^2u(x',x_N+t) \nabla H(\nabla u(x',x_N+t)), \frac{\nabla u(x',x_N+t)}{|\nabla u(x',x_N+t)|}\rangle\frac{\nabla u(x',x_N+t)}{|\nabla u(x',x_N+t)|}\\&=\nabla _{L_{\overline u,x'}}\overline H(\nabla \overline u))(x'),
        \end{split}
\end{equation}
in $x'\in\{|\nabla \overline u|\ge \varepsilon\}$.

By \cite{SZ1,SZ2} we recall that 
\begin{equation}\label{ZumbrunEuclideo}
\begin{split}
   &\sum_j |\nabla u_j(x',x_N+t)|^2-\left|\nabla _{L_{u,(x',x_N+t)}}|\nabla u(x',x_N+t)|\right|^2\\&-|\nabla|\nabla u(x',x_N+t)||^2=|\nabla u(x',x_N+t)|^2\sum_l k_{l,u^t}^2 
\end{split}
\end{equation}
where $k_{1,u^t},..., k_{(N-1),u^t}$ denote the principal curvatures of $L_{u,(x',x_N+t)}$.

Using \eqref{tang} with $H=|\cdot|$ and by \eqref{ZumbrunEuclideo} we obtain 
\begin{equation}\label{kii}
    |\nabla \overline{u}(x')|^2\sum_{l=1}^{N-2}k_{l,\overline{u}}^2(x')=\lim_{t\rightarrow +\infty} |\nabla u(x',x_N+t)|^2\sum_l k_{l,u^t}^2,
\end{equation}

where we denote with $k_{i,\overline u}(x')$ the $i-$th principal curvature of $L_{\overline u,x'}$ at $x'\in L_{\overline u,x'}\cap \{|\nabla \overline u|\ge \epsilon\}$.

 We now consider $\varphi:=\varphi_1(x')\varphi_2(x_N)$ with $\varphi_1\in C^1_c(\R^{N-1})$ and $\varphi_2\in C^1_c(\R)$. In particular, we take $\varphi_2:=\sqrt{\mu}\psi(\mu x_N)$, where $\mu >0$ and $\psi\in C^1_c(\R)$ such that 
 \begin{equation}\label{psi}
     \int_\R \psi^2(x_N)dx_N=1\quad \text{and} \quad \int_\R \varphi^2(x_N)dx_N=1.
 \end{equation}

We remark that, by Theorem \ref{finirà}, \eqref{formulafondamentale} holds for $u_t$. Using \eqref{tang}, \eqref{kii}, \eqref{psi} and  \eqref{formulafondamentale}, we have 
\begin{equation}\label{forms}
       \begin{split}
       & \int_{\{\nabla \overline u\ge \epsilon\}} \frac{B'(\overline H(\nabla \overline u))}{\overline H(\nabla \overline u)}\varphi_1^2  |\nabla \overline u|^2 \sum_{i=1}^{N-2} k_{i,\overline u}^2  +B''(\overline H(\nabla \overline u)) \left|\nabla _{L_{\overline u,x'}}\overline H(\nabla \overline u)\right|^2\varphi_1^2 \,dx'\\ &= \int_\R \int_{\{\nabla \overline u\ge \epsilon\}} \frac{B'(\overline H(\nabla \overline u))}{\overline H(\nabla \overline u)}\varphi_1^2  |\nabla \overline u|^2 \sum_{i=1}^{N-2} k_{i,\overline u}^2  +B''(\overline H(\nabla \overline u)) \left|\nabla _{L_{\overline u,x'}}\overline H(\nabla \overline u)\right|^2\varphi_1^2\varphi_2^2 \,dx'\,dx_N \\&=\lim_{t\rightarrow +\infty}\int_\R \int_{\{\nabla \overline u\ge \epsilon\}} \frac{B'( H(\nabla  u^t))}{ H(\nabla  u^t)}\varphi_1^2  |\nabla  u^t|^2 \sum_{i=1}^{N-1} k_{i, u^t}^2  \\ &\qquad\qquad\qquad\qquad+B''( H(\nabla  u^t)) \left|\nabla _{L_{ u^t,(x',x_N+t)}} H(\nabla  u^t)\right|^2\varphi_1^2\varphi_2^2 \,dx'\,dx_N \\& \le C_H \lim_{t\rightarrow +\infty} \int_{\R^N} \langle A(\nabla u^t)\nabla (\varphi_1\varphi_2),\nabla (\varphi_1\varphi_2)\rangle |\nabla u^t|^2 \\&=\int_{\R^N} \langle A(\nabla \overline u)\nabla \varphi,\nabla \varphi\rangle |\nabla \overline u|^2 \,dx'\,dx_N. 
    \end{split}
\end{equation}

where $C_H$ is a positive constant depending on $H$.

We note that 
\begin{equation}\label{spacchetamento}
\begin{split}
   \langle A(\nabla \overline u)\nabla \varphi,\nabla \varphi\rangle &=\langle \tilde A(\nabla \overline u)\nabla_{x'} \varphi,\nabla_{x'} \varphi\rangle \\&+\sum_{i=1}^N A_{in}(\nabla \overline u)\partial_{x_i}\varphi\partial_{x_N}\varphi\\&+\sum_{j=1}^{N-1} A_{nj}(\nabla \overline u)\partial_{x_j}\varphi\partial_{x_N}\varphi,
   \end{split}
\end{equation}

where $\tilde A_{ij}=A_{ij}$, for $i,j=1,...,N-1$ and $\nabla_{x'}$ is the gradient of the variables $x_1,...,x_{N-1}$. Now we prove that
\begin{equation}\label{cosanuova}
    \int_{\R^N} |\nabla \overline u|^2 \sum_{i=1}^N A_{in}(\nabla \overline u)\partial_{x_i}\varphi\partial_{x_N}\varphi \,dx=\int_{\R^N} |\nabla \overline u|^2 \sum_{j=1}^{N-1} A_{nj}(\nabla \overline u)\partial_{x_j}\varphi\partial_{x_N}\varphi \,dx=0.
\end{equation}

Since $u\in W^{1,\infty}(\R^N)$ and by definition of $\varphi$  we obtain 

\begin{equation}
    \begin{split}
       &\int_{\R^N} \sum_{j=1}^{N-1} |\nabla \overline u|^2 A_{nj}(\nabla \overline u)\partial_{x_j}\varphi\partial_{x_N}\varphi \,dx \\ & \le C \int_{\R^N}\sum_{j=1}^{N-1}  \partial_{x_j}\varphi_1(x')\varphi_2(x_N)\partial_{x_N}\varphi_2(x_N)\varphi_1(x')\,dx'\,dx_N \\&=C\sum_{j=1}^{N-1}  \int_{\R^{N-1}}\partial_{x_j}\varphi_1(x')\varphi_1(x')\,dx'\int_\R\partial_{x_N}\varphi_2(x_N)\varphi_2(x_N)\,dx_N \\& \le C\int_\R\partial_{x_N}\varphi_2(x_N)\varphi_2(x_N)\,dx_N \\&\le C\int_\R\mu^2\psi(\mu x_N)\psi'(\mu x_N)\,dx_N\\&=C\mu \int_\R \psi(y_N)\psi'(y_N)\,dy_N\rightarrow 0,
    \end{split}
\end{equation}
for $\mu\rightarrow 0$ and $C$ is a positive constant depending on $u,H$ and  $B$. In a similar way we deduce that 
\begin{equation}
    \int_{\R^N} |\nabla \overline u|^2 \sum_{i=1}^N A_{in}(\nabla \overline u)\partial_{x_i}\varphi\partial_{x_N}\varphi \,dx=0,
\end{equation}
hence \eqref{cosanuova} holds.

Taking $\epsilon$ arbitrarily small in \eqref{forms}, using \eqref{spacchetamento},  \eqref{cosanuova} and \eqref{psi}, we have 
\begin{equation}
    \begin{split}
        & \int_{\{\nabla \overline u\neq 0\}} \frac{B'(\overline H(\nabla \overline u))}{\overline H(\nabla \overline u)}\varphi_1^2  |\nabla \overline u|^2 \sum_{i=1}^{N-2} k_{i,\overline u}^2  +B''(\overline H(\nabla \overline u)) \left|\nabla _{L_{\overline u,x'}}\overline H(\nabla \overline u)\right|^2\varphi_1^2 dx'\\ &\le C\int_{\R^N} \langle \tilde A(\nabla \overline u)\nabla_{x'} \varphi,\nabla_{x'} \varphi\rangle |\nabla \overline u|^2 \,dx'\,dx_N 
        \\&=C \int_{\R^N} \langle \tilde A(\nabla \overline u)\nabla \varphi_1,\nabla \varphi_1\rangle |\nabla \overline u|^2 \,dx',
    \end{split}
\end{equation}

where $C$ is a positive constant depending on $\overline{H}.$

\end{proof}

We now recall a result (see \cite[Theorem $1.1$]{CoFaVa2}) of a pointwise energy bound in an anisotropic setting. Similar estimates were given in \cite{CGS,DG}.

\begin{lem}[ see \cite{CoFaVa2}]
Let $u\in C^{1,\alpha}_{loc}(\R^N)\cap W^{1,\infty}(\R^N)$ a weak solution of \eqref{equazione}. We set 
\begin{equation}
c_u:=\sup_{t\in [\inf u,\sup u]} F(t).
\end{equation}
Then we have 
\begin{equation}
    \Lambda_1(H(\nabla u(x)))=B'(H(\nabla u(x)))H(\nabla u(x))-B(H(\nabla u(x)))\le c_u-F(u(x)).
\end{equation}
\end{lem}

\begin{rem}\label{remarketto}
    If the set where the gradient of $u$ vanishes is empty then follows that $c_u=\max \{F(\inf u),F(\sup u)\}$. Indeed, suppose by contradiction that $c_u=F(t)$, with $t\in (\inf u,\sup u)$. Then there exist $\hat x $ such that $u(\hat x)=t,$ and by previous Lemma $\Lambda_1(H(\nabla u(\hat x)))=0$, against the assumption that the set of the gradient of $u$ is empty.

From here we suppose, for simplicity, that 
\begin{equation}\label{zio}
    c_u=F(\inf u).
\end{equation}
\end{rem}

We now are ready to give the following result of classification


\begin{lem}\label{classificazioneode2}
    Let $u\in C^{1,\alpha}_{loc}(\R^N)\cap W^{1,\infty}(\R^N)$ be a stable weak solution of \eqref{equazione}. Suppose that $\{\nabla u=0\}=\emptyset$ and $\partial_{x_N} u\ge 0$. Let $\underline u$ as in \eqref{ubassousopra} and suppose that exist $\underline w\in S^{N-2}$ and $\underline h:\R\rightarrow \R$ such that $\underline u(x')= \underline h(x'\cdot \underline w),$ for any $x'\in \R^{N-1}.$ Then one of the following possibilities holds:
    \begin{itemize}
\item [(i)] $\underline h$ is constant,
\item [(ii)] $\{\underline h '=0\}=\emptyset$,
\item [(iii)] There exist $\beta\in \R$ such that $\underline h'(t)<0$ for $t<\beta$ and $\underline h(t)=\inf u$ for $t\ge \beta$,
\item [(iv)]There exist $\beta\in \R$ such that $\underline h'(t)>0$ for $t>\beta$ and $\underline h(t)=\inf u$ for $t\le \beta$,
\item [(v)] There exist $\beta_1$ and $\beta_2$ such that $\underline h'(t)<0$ for $t<\beta_1$, $\underline h'(t)>0$ for $t>\beta_2$ and $\underline h(t)=\inf u$ for $\beta_1\le t\le \beta_2$ .
    \end{itemize}
    
\end{lem}

\begin{proof}
    We note that from Theorem \ref{analisiinfinito}, $\underline u$ is a weak solution of \eqref{equazione} in $\R^{N-1}.$ Now we will prove, since $\underline h(t)=\underline u(\underline wt)$, that $\underline h$ satisfies 
    \begin{equation}
        \left(B'(\tilde H(\underline h'(t)))\tilde H'(\underline h'(t))\right)'+f(\underline h(t))=0
    \end{equation}
    where we set $\tilde H(t)=\underline H(t\underline w).$

    Indeed, we consider the following change of variables given by $(t,y')=\psi ^{-1}(x_1,...,x_{N-1})=(\langle t,\underline w\rangle,\langle t,w^{1,\bot}\rangle,...,\langle t,w^{N-2,\bot}\rangle)$, where $\underline w,w^{1,\bot},...,w^{N-2,\bot}$ is a basis of $\R^{N-1}.$ Since $\underline u$ is a weak solution of \eqref{equazione}, we obtain 
  \begin{equation}\label{formulazionedebole247}
  \begin{split}
    &\int_{\R^{N-1}} B'(\underline H(\nabla \underline u))\langle \nabla \underline H(\nabla \underline u),\nabla \varphi_1\rangle \,dx' =\int_{\R^{N-1}}f(\underline u)\varphi_1\,dx' \\
    &  =\int_{\R^{N-1}} B'(\tilde H(\underline h'(t)))\langle \nabla \underline H(\underline h'(t)w),\nabla \varphi_1(\psi (t,y'))\rangle \,dt'\,dy' =\int_{\R^{N-1}}f(\underline h(t))\varphi_1(\psi (t,y'))\,dt'\,dy', 
    \end{split}
\end{equation}
 for any $\varphi_1\in C^1_c(\R^{N-1})$.
 
 Now we set $\tilde \varphi (t,y')= \varphi_1(\psi (t,y'))$, hence $\nabla \tilde \varphi (t,y')=D\psi(t,y')\nabla \varphi_1(\psi (t,y')),$ where $D\psi(t,y')$ denotes the Jacobian matrix of the function $\psi$. Therefore \eqref{formulazionedebole247} becomes 
 \begin{equation}\label{dani}
     \begin{split}
       &\int_{\R^{N-1}} B'(\tilde H(h'(t)))\langle \nabla \underline H(\underline h'(t)\underline w),\nabla \varphi_1(\psi (t,y'))\rangle \,dt'\,dy' =\int_{\R^{N-1}}f(\underline h(t))\varphi_1(\psi (t,y'))\,dt'\,dy' \\ 
      & =\int_{\R^{N-1}} B'(\tilde H(\underline h'(t)))\langle \nabla \underline H(\underline h'(t)\underline w),(D\psi(t,y'))^{-1}\nabla \tilde \varphi (t,y')\rangle \,dt'\,dy' \\ & =\int_{\R^{N-1}}f(\underline h(t))\tilde \varphi (t,y')\,dt'\,dy' 
      \\& = \int_{\R^{N-1}} B'(\tilde H(\underline h'(t)))\langle \nabla \underline H(\underline h'(t)\underline w),\underline w \rangle  \partial_t \tilde\varphi (t,y')\,dt'\,dy'\\+
      & \int_{\R^{N-1}} B'(\tilde H(\underline h'(t)))\langle \nabla \underline H(\underline h'(t)\underline w),(\hat D \psi(t,y'))^{-1}\nabla_{y'} \tilde \varphi (t,y')\rangle \,dt'\,dy'=\int_{\R^{N-1}}f(\underline h(t))\tilde \varphi(t,y')\,dt'\,dy',
     \end{split}
 \end{equation}
where we set $(\hat D \psi(t,y'))^{-1})_{ij}:=(D \psi(t,y'))^{-1})_{ij}$, for $i=1,...,N-1$, $j=2,...,N-1$ and $\nabla _{y'}$ denotes the gradient of the variables $y'$. Now we consider $\tilde \varphi (t,y')= \tilde \varphi_1(t)\tilde \varphi_2(y')$, with $\tilde \varphi_1\in C^1_c(\R)$ and $\tilde \varphi_2 \in C^1_c(\R^{N-2})$, satisfying 
\begin{equation}\label{cazzi}
    \int_{\R^{N-2}}\tilde \varphi_2(y')\,dy'=1 \qquad \text{and}\qquad  \int_{\R^{N-2}}\nabla \tilde \varphi_2(y')\,dy'=(0,\cdots,0).
\end{equation}
Note that the last property follows from the fact that $\tilde{\varphi}_2$ is compactly supported in $\R^{N-2}$. Since $(\hat{D} \psi(t,y'))^{-1}$ is constant, by \eqref{dani} and \eqref{cazzi} we have that 
    \begin{equation}
        \int_\R B'(\tilde H(\underline h'(t))) \tilde H'(\underline h'(t)))\tilde \varphi_1'(t) \,dt=\int_\R f(\underline h(t))\tilde \varphi_1(t)\,dt
    \end{equation}
    for all $\tilde \varphi_1\in C^1_c(\R).$

Additionally, we have: 
\begin{equation}\label{dainelli}
    \underline h(t)< \sup u
\end{equation}
for any $t\in \R$ because if $\underline h(t)=\sup u$ for some $t$, since $\partial _{x_N} u\ge 0$, then 
$$\sup u=\underline h(t)=\underline u(\underline w t)=\lim_{s\rightarrow -\infty} u(\underline wt,s)\le u(\underline w t,0)\le \sup u,$$
so $$\sup u=u(\underline wt,0),$$ which implies $\nabla u(\underline wt,0)=0$, contradicting the fact that $\{\nabla u=0\}=\emptyset.$ Next, we use the classification from Lemma \ref{classificazioneode}: if $\underline h$ satisfies conditions $(i)$ or $(ii)$ of the Lemma \ref{classificazioneode}, we fall into cases $(i)$ or $(ii)$, and we are done. Now, we show that case $(iii)$ of Lemma \ref{classificazioneode} is impossible in this situation. Indeed, if case $(iii)$ were to hold, then by the fact that $\partial_{x_N} u\ge 0$ and $\{\nabla u =0\}=\emptyset$, we have $$F(\underline h(\beta_1))=F(\underline h(\beta_2))=F(\inf \underline h)=F(\inf \underline u)=F(\inf u).$$ 
By \eqref{zio} we obtain $$F(\underline h(\beta_1))=F(\underline h(\beta_2))=c_u,$$ so from Remark \ref{remarketto} we deduce that $$\underline h(\beta_1),\underline h(\beta_2)\in \{\inf u,\sup u\}.$$ By \eqref{dainelli}, we get that $\underline h(\beta_1)=\underline h(\beta_2)=\inf u$, which contradicts the fact that $\underline h$ is strictly monotone in $(\beta_1,\beta_2)$.

Thus, the only remaining possibility is that $\underline h$ satisfies condition $(iv)$ of Lemma \ref{classificazioneode}. By changing $t$ to $-t$ if necessary, we can consider the case where the interval in $(iv)$ of Lemma \ref{classificazioneode} is of the form $(-\infty,\beta).$ By the fact that $\partial_{x_N} u\ge 0$, $\{\nabla u =0\}=\emptyset$, and by \eqref{zio} we have $$F(\underline h(\beta))=F(\inf \underline u)=F(\inf u)=c_u.$$ Thus by Remark \ref{remarketto} and \eqref{dainelli} we get $\underline h(\beta)=\inf u.$  Now, if $\underline h(t)=\inf u$ for $t\ge \beta$, then we are in case $(iii)$. Thus, we can assume there exists some $t^*>\beta$ such that $\underline h(t^*)>\inf u$. This implies there must be some $t'>\beta$ where $h'(t')> 0$. Consequently, there exists an interval $(\beta_2, \beta_3)$, as large as possible, with $\beta_2 \ge \beta=:\beta_1$ and $\beta_3\in (\beta_2,+\infty)\cup \{+\infty\}$ such that $h'(t)>0$ for any $t \in (\beta_2, \beta_3)$.

If $\beta_3 \neq +\infty$ then $h'(\beta_2)=h'(\beta_3)=0$, and by \eqref{remarkino}, we would revert to case $(iii)$ of Lemma \ref{classificazioneode}, which we have already shown to be impossible.

Therefore, $\beta_3=+\infty$. Similarly, $\underline h(t)$ must equal to $\inf u$ in $[\beta_1, \beta_2]$, because if not, there would exist an interval $(\beta_1', \beta_2') \subset [\beta_1, \beta_2]$ such that $\underline h'(t)\neq 0$ in $(\beta'_1,\beta'_2)$ and $\underline h'(\beta'_1)=\underline h'(\beta'_2)=0$, leading us back to the impossible case $(iii)$. This shows that $\underline h'(t)\neq 0$, for $t<\beta_1$ and $t>\beta_2$ and $\underline h(t)=\inf u$ for $t\in [\beta_1,\beta_2]$ resulting in case $(v)$.

\end{proof}

Given $v: \R^N\rightarrow \R,$ with $|\nabla v|\in L^{\infty}(\R^N)$, we consider, for $R>0$, the energy 

$$ E_R(v):=\int_{B_R} \Lambda _2(H(\nabla v))-F(v).$$ 

We give the following result that controls the energy of a solution $v$ with the one $v^t$, up to a term of order $R^{N-1}$.

\begin{lem}\label{energia}
Let $u\in C^2(\R^N)\cap W^{1,\infty}(\R^N)$ be a weak solution of \eqref{equazione}, with $\partial_{x_N} u\ge 0$ and $\{\nabla u=0\}=\emptyset.$ Then exists $C>0$, depending only on $N, u, H$ and $B$ such that 
\begin{equation}
    E_R(u)\le E_R(u^t)+CR^{N-1}
\end{equation}
for any $t\in \R.$
\end{lem}

\begin{proof}
    Since $u\in C^2(\R^N),$ we have that $u$, therefore $u^t$, is a strong solution of \eqref{equazione}. Integrating by parts we have
\begin{equation}\label{attaccabro}
    \begin{split}
        \partial_t E_R(u^t)&=\int_{B_R} B'(H(\nabla u^t))\langle \nabla H(\nabla u^t),\partial_t(\nabla u^t)\rangle - f(u^t)\partial_t u^t \\ &=\int_{B_R} B'(H(\nabla u^t))\langle \nabla H(\nabla u^t),\partial_t(\nabla u^t)\rangle +\operatorname{div}\left(B'(H(\nabla u^t))\nabla H (\nabla u^t)\right)\partial_t u^t \\ &=\int_{\partial B_R} B'(H(\nabla u^t))\partial_t u^t\langle\nabla H(\nabla u^t),\eta\rangle d\sigma,
    \end{split}
\end{equation}
    where $\eta$ is the exterior normal of $B_R$.

Since $\nabla H$ is a $0-$homogeneous function and $B'(t)\in L^{\infty}_{loc}([0,+\infty))$, by \eqref{attaccabro}, and since $\partial_{x_N}u\ge 0$, we get
\begin{equation}
\begin{split}
    E_R(u^t)-E_R(u)&=\int_0^t \partial _sE_R(u^s)ds \ge -C \int_{\partial B_R}\int_0^t \partial _su^sd\sigma ds\\&=-\tilde C\int_{\partial B_r} (u^t-u)d\sigma \ge -2\tilde C\|u\|_{L^{\infty}(\R^N)}|\partial B_r|\\ &\ge -CR^{N-1},
    \end{split}
\end{equation}
where $C$ is a positive constant depending on $B$, $H$, $N$ and $u$.
\end{proof}

\begin{lem}\label{energy}
Let $u\in C^{1,\alpha}_{loc}(\R^N)\cap W^{1,\infty}(\R^N)$ be a stable weak solution of \eqref{equazione}. Suppose that $\{\nabla u=0\}=\emptyset$ and $\partial_{x_N} u\ge 0$. Let $\underline u$ has one-dimensional symmetry. Then 
\begin{equation}
    \int_{B_R}\Lambda_2(H(\nabla u))-F(u)+c_u dx \le CR^{N-1}
\end{equation}
where $C$ is a positive constant depending on $B$, $H$, $N$ and $u$
\end{lem}

\begin{proof}
We suppose that there exist $\underline w\in S^{N-2}$ and $\underline h:\R\rightarrow \R$ such that $\underline u(x')= \underline h(x'\cdot w),$ for any $x'\in \R^{N-1}.$ Recalling \eqref{ubassousopra}, we set $\tilde H(t):=\underline H(t w):=H(tw,0).$

    First, we will show that  
    \begin{equation}\label{miao}
        \int_{\R} \Lambda_2(\tilde H(\underline h'(t)))-F(\underline h(t))+c_u dt <+\infty.
    \end{equation}
Indeed, since $\partial_{x_N}u\ge 0$, we remark that $c_u=F(\inf u)=F(\inf \underline h)$. By Lemma \ref{basta2}, by $(h_B)$-$(iii)$ and by \eqref{zio}, we get 

\begin{equation}
\begin{split}
    &\int_{-\infty}^{+\infty}c_u-F(\underline h(t)) \,dt =\int_{-\infty}^{+\infty}\Lambda_1(\tilde H(\underline h'(t)))\,dt \\&=\int_{-\infty}^{+\infty}\int_0^{\tilde H(\underline h'(t))} B''(s)s \,ds\,dt \le \hat C\int_{-\infty}^{+\infty}\int_0^{\tilde H(\underline h'(t))} B'(s) \,ds\,dt \\&=\hat C\int_{-\infty}^{+\infty}\Lambda_2(\tilde H(\underline h'(t))) \,dt,
    \end{split}
\end{equation}

where $\hat C$ is a positive constant depending on $B$. Now, \eqref{miao} is proved, if 
\begin{equation}\label{serrin}
    \int_{-\infty}^{+\infty}\Lambda_2(\tilde H(\underline h'(t))) \,dt< +\infty.
\end{equation}

To prove this, we consider different cases in Lemma \ref{classificazioneode2}. We consider case (v) in Lemma \ref{classificazioneode2}, the other cases are similar. Since $\underline h\in W^{1,\infty}(\R)$ and $B'(t)\in L^{\infty}_{loc}[0,+\infty)$, using the fact that $H$ is a norm equivalent to Euclidean one (see \eqref{H equiv euclidea}) we get 
\begin{equation}
\begin{split}
&\int_{-\infty}^{+\infty}\Lambda_2(\tilde H(\underline h'(t))) \,dt=\int_{-\infty}^{+\infty}  \int_0^{\tilde H(\underline h'(t))}B'(s)\,ds \\ & \le \|B'\|_{L^{\infty}([0,\tilde H(\underline h'(t))])}\alpha_2\int_{-\infty}^{+\infty}|\underline h'(t)|\,dt \\ & =\|B'\|_{L^{\infty}([0,\tilde H(\underline h'(t))])}\alpha_2\left(\lim_{\alpha\rightarrow-\infty}\int_{\alpha}^{\beta_1}-h'(t)+\lim_{\beta\rightarrow\infty}\int_{\beta_2}^{\beta}h'(t)\right)\\&=\|B'\|_{L^{\infty}([0,\tilde H(\underline h'(t))])}\alpha_2\left(\lim_{\alpha\rightarrow-\infty}h(\alpha)-h(\beta_1)+\lim_{\beta\rightarrow\infty} h(\beta)-h(\beta_2)\right) \\&\le 4\|B'\|_{L^{\infty}([0,\tilde H(\underline h'(t))])}\alpha_2\|u\|_{L^{\infty}(\R^N)}<+\infty,
\end{split}
\end{equation}
and \eqref{miao} is done.

Therefore 
\begin{equation}\label{cosa}
\begin{split}
&\int_{B_R}\Lambda_2(\underline H(\nabla \underline u))-F(\underline u)+c_u dx \\&= \int_{B_R} \Lambda_2(\tilde H(\underline h'(t)))-F(h(t))+c_u dx'' dt dx_N \le \tilde CR^{N-1},
   \end{split}
\end{equation}
where $\tilde C$ is a positive constant depending on $B,H,$ and $u$. 
By Lemma \ref{energia} and \eqref{cosa} we get
\begin{equation}
    \begin{split}
        &\int_{B_R}\Lambda_2(H(\nabla u))-F(u)+c_u dx =E_R(u)+c_u|B_R| \\ &\le \lim_{t\rightarrow +\infty} E_R(u^t)+c_u|B_R|+ CR^{N-1} \\& \le \int_{B_R}\Lambda_2(\underline H(\nabla \underline u))-F(\underline u)+c_u dx+CR^{N-1} \le CR^{N-1},
    \end{split}
\end{equation}
where $C$ is a positive constant depending on $B,H,N$ and $u.$

   \end{proof}

\section{Proof of Theorem \ref{DegiorgiN=3}}\label{dimostrazioneteorema}

We start this section, showing that the weak solutions of \eqref{equazione} that are monotone in one variable are stable.

\begin{lem}\label{monotoniaimplicastabilità}
    Let $u\in C^{1}(\R^N)$, be a weak solution of \eqref{equazione}. Suppose $\partial_{x_N}u\ge 0$. Then \begin{equation*}
    \int_{\R^N}\langle A(\nabla u)\nabla \varphi,\nabla \varphi\rangle-\int_{\{\partial _{x_N} u>0\}}f'(u)
    \varphi^2 \ge 0,
\end{equation*}
for any $\varphi\in C^1_c(\R^N\setminus\{\nabla u=0\})$.

In particular, if $u\in C^{1}(\R^N)$ is a weak solution of \eqref{equazione}, with  $\partial_{x_N}u> 0$, then $u$ is stable.
\end{lem}

\begin{proof}
     Fixed $\delta >0$, let us consider \begin{equation}\label{daniel}
    \psi:=\frac{\varphi^2}{u_N+\delta}
     \end{equation}
for $\varphi\in C^1_c(\R^N\setminus\{\nabla u=0\}).$ Since $\partial_{x_N}u\ge 0$, $\varphi\in C^1_c(\R^N\setminus\{\nabla u=0\})$ and using \eqref{daniel} in \eqref{u_irisolvelinearizzato} we get

\begin{equation}\label{ii}
        \begin{split}
        &\int_{\R^N} B''(H(\nabla u))\langle \nabla H(\nabla u),\nabla u_N\rangle\langle \nabla H(\nabla u),\nabla \varphi\rangle\frac{2\varphi}{u_N+\delta} \\ &- B''(H(\nabla u))\langle \nabla H(\nabla u),\nabla u_N\rangle\langle \nabla H(\nabla u),\nabla u_N\rangle\frac{\varphi^2}{(u_N+\delta)^2}  \\&+ B'(H(\nabla u)) \langle D^2H(\nabla u)\nabla u_N,\nabla \varphi\rangle\frac{2\varphi}{u_N+\delta}\\ &-B'(H(\nabla u)) \langle D^2H(\nabla u)\nabla u_N,\nabla u_N\rangle\frac{\varphi^2}{(u_N+\delta)^2}  =\int_{\R^N} f'(u)u_N\frac{\varphi^2}{u_N+\delta} .
    \end{split}
    \end{equation}

Recalling the definition of the matrix $A$ (see \eqref{matrice}), we deduce 

\begin{equation}\label{iii}
        \begin{split}
    0=&\int_{\R^N} B''(H(\nabla u))\langle \nabla H(\nabla u),\nabla u_N\rangle\langle \nabla H(\nabla u),\nabla \varphi\rangle\frac{2\varphi}{u_N+\delta} \\ &- B''(H(\nabla u))\langle \nabla H(\nabla u),\nabla u_N\rangle\langle \nabla H(\nabla u),\nabla u_N\rangle\frac{\varphi^2}{(u_N+\delta)^2}  \\&+ B'(H(\nabla u)) \langle D^2H(\nabla u)\nabla u_N,\nabla \varphi\rangle\frac{2\varphi}{u_N+\delta}\\ &-B'(H(\nabla u)) \langle D^2H(\nabla u)\nabla u_N,\nabla u_N\rangle\frac{\varphi^2}{(u_N+\delta)^2}  -f'(u)u_N\frac{\varphi^2}{u_N+\delta} \\ &=\int_{\R^N}-\langle A(\nabla u)\left(\nabla \varphi-\frac{\varphi}{u_N+\delta}\nabla u_N\right),\left(\nabla \varphi\frac{\varphi}{u_N+\delta}\nabla u_N\right)\rangle\\ &\qquad +\langle A(\nabla u)\nabla \varphi,\nabla \varphi\rangle-f'(u)u_N\frac{\varphi^2}{u_N+\delta}. 
    \end{split}
    \end{equation}

Since $A$ is definite positive by Lemma \ref{matricedefinitapositiva}, for $\delta$ sufficiently small we obtain
\begin{equation}
    \int_{\R^N}\langle A(\nabla u)\nabla \varphi,\nabla \varphi\rangle-\int_{\{u_N>0\}}f'(u)
    \varphi^2 \ge 0,
\end{equation}

and therefore the thesis.
 

\end{proof}

Given a function $v\in C^{1,\alpha}_{loc}(\R^N)$, let us consider the graph of a function $Y(x):=(x,v(x))\in \R^N\times \R.$ The following result is useful for the proof of Theorem \ref{DegiorgiN=3}. 

\begin{lem}\label{lemmaaaa}
Suppose $v\in C^{1,\alpha}_{loc}(\R^N)$ satisfies 
\begin{equation}\label{proprieta}
    \int_{|Y(x)|\le r} B'(H(\nabla v))H(\nabla v) dx \le Cr^2
\end{equation}
for $r$ sufficiently large and $C$ positive constant. Then we have 
\begin{equation}
    \int_{\sqrt{r}\le |Y|\le r} \frac{B'(H(\nabla v))H(\nabla v)}{|Y|^2}\le C \ln r
\end{equation}
for $r$ sufficiently large and $C>0$.
\end{lem}

\begin{proof}
      Since 
      \begin{equation}
2\int_{|Y|}^{r}s^{-3}\,ds=|Y|^{-2}-r^{-2},
      \end{equation}
using the Fubini theorem and by \eqref{proprieta}, we get 
\begin{equation}\label{effepi}
\begin{split}
   &\int_{\sqrt{r}\le |Y|\le r} B'(H(\nabla v))H(\nabla v)(|Y|^{-2}-r^{-2}) \\ &= 2\int_{|Y|}^{r}\int_{\sqrt{r}\le |Y|\le r} B'(H(\nabla v))H(\nabla v)s^{-3}\, ds\,dx \\&=2 \int_{\sqrt r}^r s^{-3}\int_{\sqrt r\le |Y|\le s} B'(H(\nabla v))H(\nabla v)dxds \\ &\le C \int_{\sqrt r}^r s^{-1} ds \le C\ln r
   \end{split}
    \end{equation}
    if $r$ is large enough.
By \eqref{effepi}, using the fact that $|Y|^{-2}=(|Y|^{-2}-r^{-2})+r^{-2}$ and \eqref{proprieta} we have 

\begin{equation}\label{effepiu}
\begin{split}
   &\int_{\sqrt{r}\le |Y|\le r} B'(H(\nabla v))H(\nabla v)|Y|^{-2} \\ &\le  \int_{\sqrt{r}\le |Y|\le r} B'(H(\nabla v))H(\nabla v)r^{-2}+C\ln r \\& \le C+C\ln r,
   \end{split}
    \end{equation}
where $C$ is a positive constant.

\end{proof}

\begin{proof}[Proof of Theorem \ref{DegiorgiN=3}]
Since $\underline u$ and $\overline u$ are functions on $\R^2$ using  \eqref{formulafondamentale2}, we have 
$$k_{1,\underline u}=k_{1,\overline u}=0=\nabla _{L_{\underline u,x'}}\underline H(\nabla \underline u)=\nabla _{L_{\overline u,x'}}\overline H(\nabla \overline u).$$
Therefore $\underline u$ and $\overline u$ have one-dimensional symmetry.

Now we prove that $u$ satisfies \eqref{proprieta}. By \eqref{definizioni} and by assumptions $(h_B)$-$(iii)$  we remark that
\begin{equation}\label{daniel2}
    B'(t)t-\Lambda_2(t)=\Lambda_1(t),\qquad \text{and}\qquad \Lambda_1(t)\le C(B)\Lambda_2(t)
\end{equation}
for any $t\ge 0$, and where $C(B)$ is a positive constant depending on $B$.
Using
Lemma \ref{energy} and by \eqref{daniel2}, we get 
\begin{equation}\label{a}
\begin{split}
    &\int_{|(x,u(x))|\le R} B'(H(\nabla u))H(\nabla u) \\ &\le \int_{B_R} B'(H(\nabla u))H(\nabla u) \le C(B) \int_{B_R} \Lambda_2(H(\nabla u)) \\ & \le C(B)\int_{B_R}\Lambda_2(H(\nabla u))-F(u)+c_u \le C(B,H,N,u) R^2,
    \end{split}
\end{equation}
where $C(B,H,N,u)$ is a positive constant depending on $B,H,N$ and $u$.

Since $\nabla u\in L^{\infty}(\R^N)$, we have 
\begin{equation}\label{b}
H(\nabla u)^4B'(H(\nabla u))\le C(u,H) B'(H(\nabla u))H(\nabla u),
\end{equation}
where $C(u,H)$ is a positive constant depending on $u$ and $H$.

Fixed $R>0$, we now consider the function
\begin{equation}\label{phi}
    \varphi_R(x):=\left\{\begin{array}{llll}
                1 &\text{if } |Y|<\sqrt{R}\\
                \frac{2\ln(R/|Y|)}{\ln R} &\text{if } \sqrt{R}<|Y|<R\\
                 0 \quad &\text{if } |Y|\ge R.\\
                 \end{array}\right.
\end{equation}

By construction and using the fact that $H$ is a norm equivalent to euclidean norm we have 
\begin{equation}\label{gradient}
    |\nabla \varphi_R(x)|^2 \le \tilde C_H\frac{|x|^2+u(x)^2H(\nabla u)^2}{|Y|^4(\ln R)^2},
\end{equation}
where $\tilde C_H$ is a positive constant depending on $H$.

Since $H$ is $1$-homogeneous function, there exist positive constants $M,\overline{M}$ such that 
\begin{equation}\label{D2H}
    |\nabla H(\xi)|\le M,\qquad|D^2H(\xi)|\le \frac{\overline M}{H(\xi)}\quad\forall \xi\in \R^N\setminus \{0\}.
\end{equation}

Taking \eqref{phi} in \eqref{formulafondamentale}, by \eqref{a}, \eqref{b}, \eqref{gradient}, \eqref{D2H}, by Lemma \ref{lemmaaaa} and by assumptions $(h_B)$-$(iii)$, we get
\begin{equation}
       \begin{split}
       & \int_{\{\nabla u\neq 0\}\cap B_{R}} \frac{B'(H(\nabla u))}{H(\nabla u))}  |\nabla u|^2 \sum_{i=1}^{N-1} k_{i,u}^2  +B''(H(\nabla u)) \left|\nabla _{L_{u,x}}H(\nabla u)\right|^2\\ &\le C_H\int_{\R^3} \langle A(\nabla u)\nabla \varphi,\nabla \varphi\rangle|\nabla u|^2 \\ &=C_H\int_{\R^3} \left(B''(H(\nabla u))  \langle \nabla H(\nabla
u), \nabla \varphi\rangle^2 +B'(H(\nabla
u))\langle D^2H(\nabla u)\nabla \varphi,\nabla \varphi \rangle\right)|\nabla u|^2\\&\le C_H\int_{\R^3} \left(M^2B''(H(\nabla u)) +\overline M\frac{B'(H(\nabla
u))}{H(\nabla u)}\right)|\nabla \varphi|^2|\nabla u|^2 \\ &\le C(H,B)\int_{B_R\setminus B_{\sqrt{R}}} \frac{H(\nabla u)^2B'(H(\nabla u))|x|^2}{|Y|^4(\ln R)^2}+ C(H,B)\int_{B_R\setminus B_{\sqrt{R}}} \frac{H(\nabla u)^4B'(H(\nabla u))|u|^2}{|Y|^4(\ln R)^2}\\ &\le C(H,B)\int_{B_R\setminus B_{\sqrt{R}}} \frac{H(\nabla u)B'(H(\nabla u))}{|Y|^2(\ln R)^2}+ C(H,B,u)\int_{B_R\setminus B_{\sqrt{R}}} \frac{H(\nabla u)B'(H(\nabla u))}{|Y|^2\ln R^2} \\ &\le C(H,B,N,u)\frac{\ln R}{(\ln R)^2},
    \end{split}
\end{equation}
where $C(H,B,N,u)$ is a positive constant depending on $H,B,N$ and $u$.
For $R$ arbitrarily large we obtain that $u$ possesses one-dimensional symmetry. 
\end{proof}

\section*{Acknowledgements}
We would like to thank Prof. Alberto Farina for the very useful conversations on the
paper.

\section*{Data availability statement}
All data generated or analyzed during this study are included in this published article.

\end{document}